\pdfoutput=1
\RequirePackage{ifpdf}
\ifpdf 
\documentclass[pdftex]{sigma}
\else
\documentclass{sigma}
\fi

\usepackage{enumerate}

\numberwithin{equation}{section}

\newtheorem{Theorem}{Theorem}[section]
\newtheorem{Corollary}[Theorem]{Corollary}
\newtheorem{Question}[Theorem]{Question}
\newtheorem{Lemma}[Theorem]{Lemma}
\newtheorem{Proposition}[Theorem]{Proposition}
\newtheorem{Claim}[Theorem]{Claim}
 { \theoremstyle{definition}
\newtheorem{Definition}[Theorem]{Definition}

\newtheorem{Example}[Theorem]{Example}
\newtheorem{Notation}[Theorem]{Notation}
\newtheorem{Remark}[Theorem]{Remark} }

\newcommand{\cat}[1]{\textsc{#1}}

\begin{document}

\newcommand{\arXivNumber}{2008.10072}

\renewcommand{\PaperNumber}{030}

\FirstPageHeading

\ShortArticleName{Prescribed Riemannian Symmetries}

\ArticleName{Prescribed Riemannian Symmetries}

\Author{Alexandru CHIRVASITU}

\AuthorNameForHeading{A.~Chirvasitu}

\Address{Department of Mathematics, University at Buffalo, Buffalo, NY 14260-2900, USA}
\Email{\href{mailto:achirvas@buffalo.edu}{achirvas@buffalo.edu}}

\ArticleDates{Received September 27, 2020, in final form March 10, 2021; Published online March 25, 2021}

\Abstract{Given a smooth free action of a compact connected Lie group $G$ on a smooth compact manifold $M$, we show that the space of $G$-invariant Riemannian metrics on $M$ whose automorphism group is precisely $G$ is open dense in the space of all $G$-invariant metrics, provided the dimension of $M$ is ``sufficiently large'' compared to that of $G$. As~a~con\-se\-qu\-ence, it follows that every compact connected Lie group can be realized as the automorphism group of some compact connected Riemannian manifold; this recovers prior work by Bedford--Dadok and Saerens--Zame under less stringent dimension conditions. Along the way we also show, under less restrictive conditions on both dimensions and actions, that the space of~$G$-invariant metrics whose automorphism groups preserve the $G$-orbits is dense $G_{\delta}$ in the space of all $G$-invariant metrics.}


\Classification{53B20; 58D17; 58D19; 57S15}

\section{Introduction}

The present paper fits into the general theme of realizing a predetermined group as the symmetry group of a structure (combinatorial, topological, geometric, etc.) of given type. Variations on~the theme abound in the literature, mostly (but by no means exclusively) in the context of {\it finite} groups. To list a few instances:

\begin{enumerate}[(1)]\itemsep=0pt
\item Every finite group is the automorphism group of
 \begin{itemize}\itemsep=0pt
 \item a finite graph~\cite{frcht},
 \item even better, a finite 3-regular graph~\cite[Theorems 2.4 and 4.1]{frcht-3},
 \item more generally, a finite graph with given connectivity or chromatic number, regular of given degree, and a number of other such constraints~\cite[Theorem~1.2]{sabid},
 \item some convex polytope, with the group acting either purely combinatorially~\cite[Theorems 1 and 2]{Schulte:2015eu} or isometrically~\cite[Theorem~1.1]{Doignon:2018ch}.
 \end{itemize}
\item More generally, {\it arbitrary} (possibly infinite) groups are graph isomorphism groups~\cite[Theo\-rem]{sabid-inf}.
\item Switching to the more topologically-flavored setup that informs this paper,
 \begin{itemize}\itemsep=0pt
 \item {\it Polish} (i.e.,~separable completely metrizable) topological groups can be realized as isometry groups of separable complete metric spaces (\cite{gk} and~\cite[Section~3]{mell}),
 \item in the same spirit, compact metrizable groups are isometry groups of compact metric spaces~\cite[Theorem~1.2]{mell},
 \item the same goes for {\it locally} compact groups and spaces~\cite[Theorem~2.1]{ml},
 \item Lie groups are isometry groups of manifolds (equipped with metrics compatible with the manifold structure)~\cite[Corollary~1.2 and discussion following Proposition~1.4]{niem}.
 \end{itemize}
\end{enumerate}
This paper was originally motivated by the following problem, posed as open in~\cite[Section~4]{mell} (and appearing also as~\cite[Question~Q4]{niem}):

\begin{Question}\label{qu:mot}
 Is it the case that every compact Lie group is the isometry group of some compact Riemannian manifold?
\end{Question}

The answer turns out to be affirmative, as confirmed in both~\cite[Theo\-rem~3]{bd} and~\cite[Theo\-rem~1]{sz} (I am grateful to one of the referees for bringing this work to my attention; I was not aware of it while writing an early draft).

With this question as the initial motivating driver, one can then pose the natural follow-up problem (to be made precise) of ``how large'' the set of desirable Riemannian metrics is, given a~manifold equipped with an action by a compact Lie group. The results below all revolve around the general principle that given an (always smooth, for us) action of $G$ on $M$, ``most'' $G$-invariant Riemannian metrics on $M$ are maximally rigid. In the case $G=\{1\}$ this same generic rigidity principle informs~\cite[Corollary~8.2 and Proposition~8.3]{eb}, which we paraphrase slightly as

\begin{Theorem}
 Let $M$ be a compact smooth manifold. The set of Riemannian metrics on $M$ with trivial isometry group is open in the space of all Riemannian metrics, and it is dense if all connected components of $M$ have dimension $\ge 2$.
\end{Theorem}

The version of this result proven below (Theorems~\ref{th:equirig} and~\ref{th:precaut}), involving a $G$-action, reads as follows (with $\mathcal{M}_G(M)$ denoting the space of $G$-invariant metrics on $M$, equipped with its $C^{\infty}$ topology):

\begin{Theorem}\label{th:summ}
 Let $G$ be a compact Lie group acting smoothly on the compact manifold $M$.
 \begin{enumerate}[$(a)$]\itemsep=0pt
 \item\label{item:21} If
 \begin{itemize}\itemsep=0pt
 \item the action is principal $($Definition~$\ref{def:princ})$, and
 \item all connected components of $M$ have dimension $\ge 3+\dim G$
 \end{itemize}
 then the space of $G$-invariant Riemannian metrics whose isometry groups leave all $G$-orbits invariant is a dense $G_{\delta}$ subset of $\mathcal{M}_G(M)$.
 \item\label{item:22} If furthermore
 \begin{itemize}\itemsep=0pt
 \item $G$ is connected,
 \item the action is free, and
 \item all connected components of $M$ also have dimension $\ge 2\dim G+1$
 \end{itemize}
 then the space of $G$-invariant metrics on $M$ whose isometry group is precisely $G$ is open dense in $\mathcal{M}_G(M)$.
 \end{enumerate}
\end{Theorem}

In Section~\ref{se:prel} we gather a number of preliminary remarks of use in the sequel (on topology, Riemannian geometry, etc.).

Section~\ref{se:vert} revolves around {\it vertical} Riemannian metrics: given an action of $G$ on $M$, these are the $G$-invariant metrics whose isometry groups preserve the $G$-orbits (as sets, not necessarily pointwise). The term is inspired by the theory of fibrations / submersions (e.g.,~\cite[Chapter~9]{bes}): the $G$-orbits are the fibers of the fibration $M\to M/G$, so the vectors tangent to the orbits are vertical in fibration-specific terminology. The main result in that section is Theorem~\ref{th:equirig}, matching ({\it \ref{item:21}}) of Theorem~\ref{th:summ} above.

Naturally, if a $G$-invariant metric on $M$ is {\it maximally rigid} in the sense that its isometry group is precisely $G$ (and not larger), it will also be vertical. For that reason, Section~\ref{se:vert} serves as preparation for Section~\ref{se:maxrig}. In the latter we focus on maximally rigid $G$-invariant metrics on~$M$. Here the main result is Theorem~\ref{th:precaut}, corresponding to part ({\it\ref{item:22}}) of Theorem~\ref{th:summ} above.

As for Question~\ref{qu:mot}, Corollary~\ref{cor:isiso} below recovers that affirmative answer for {\it connected} compact Lie groups. There are trade-offs as compared to~\cite[Theorem~3]{bd} and~\cite[Theorem~1]{sz}: while the latter make no connectedness assumptions, the Riemannian manifolds obtained here (with prescribed symmetry group) tend to have smaller dimension.

\section{Preliminaries}\label{se:prel}

We will need some background on Riemannian geometry, as covered well in numerous sources:~\cite{berg-pan,bes,doC92,kn1,pet} will do for instance (as does~\cite[Appendix A]{ck-ric} for a quick reference), and we cite some of these more precisely in the discussion below. We use some of the standard conventions. Having fixed a coordinate patch of the Riemannian manifold $M$ with coordinates
\begin{gather*}
 x^i,\qquad 1\le i\le n,\qquad n:=\dim M
\end{gather*}
we denote
\begin{itemize}\itemsep=0pt
\item by $\delta_{ij}$ the Kronecker delta, $1$ for $i=j$ and $0$ otherwise,
\item by $g_{ij}$ the Riemannian metric tensor,
\item by $g^{k\ell}$ the inverse of $g_{ij}$,
\item by $R_{ijk}^{\ell}$ the curvature $(3,1)$-tensor,
\item by
 \begin{gather}
 R_{ik} = R_{ijk}^j\label{eq:ric20}
 \end{gather}
 the Ricci $(2,0)$-tensor, with Einstein summation convention (i.e.,~summing over the repe\-ated $j$ in (\ref{eq:ric20})),
\item by $R$ (unadorned) the scalar curvature, $\operatorname{tr} R_{ij}$, a function on $M$,
\item on one occasion, by $Z_{ij}$ the {\it traceless} Ricci $(2,0)$-tensor
 \begin{gather}\label{eq:trric}
 Z_{ij} = R_{ij} - \frac{1}{n} R g_{ij}.
 \end{gather}
\end{itemize}
See for instance~\cite[Chapter~1]{bes} or~\cite[Chapter~3]{pet} for a recollection of the various notions.

We also adopt the usual convention on raising and lowering indices via $g^{k\ell}$ and $g_{ij}$ respectively, for instance as in~\cite[Section~1.42]{bes}: for a tensor $A_{\cdots}^{j -}$ we set
\begin{gather*}
 A_{i\cdots}^{-}:=g_{ij}A_{\cdots}^{j -}
\end{gather*}
(summation over the repeated $j$, as always). Analogous formulas hold for raising rather than lowering an index, with $g^{k\ell}$ in place of $g_{ij}$. We will refer again, for instance, to the Ricci $(1,1)$-tensor
\begin{gather*}
 R^i_j = g^{ik}R_{kj}
\end{gather*}
(see~\cite[Remark 1.91]{bes}). It induces an operator on each tangent space $T_pM$, $p\in M$ of a Riemannian manifold, and that operator is self-adjoint (i.e.,~symmetric with respect to the Hilbert space structure on $T_pM$ imposed by the Riemannian metric). The symmetry, concretely, simply means that
\begin{gather*}
 R^i_j = R_i^j.
\end{gather*}

``Smooth'' always means $C^{\infty}$. The main manifold under consideration (typically denoted by~$M$) can be assumed boundary-less, but various auxiliary submanifolds thereof will, in general, have boundaries or be non-compact: in those cases the arguments will be local in nature, so the non-compactness and/or presence of a boundary will not make a difference.

For a smooth manifold $M$, we follow~\cite{eb} in denoting by $\mathcal{M}:=\mathcal{M}(M)$ the space of smooth Riemannian structures on $M$. In the presence of a smooth action of a (typically compact) Lie group $G$ on $M$ we amplify this notation by writing
\begin{gather*}
 \mathcal{M}_G:=\mathcal{M}_G(M)
\end{gather*}
for the space of (always smooth) $G$-invariant Riemannian metrics on $M$.

The following piece of terminology is justified by the example of a fibration $M\to M/G$ induced by a free $G$-action on $M$, with fibers $\cong G$ regarded as ``vertical'' in a pictorial rendition of that fibration.

\begin{Definition}\label{def:vert}
 Let $G$ be a Lie group acting smoothly on the smooth manifold $M$. Vectors in~$TM$ tangent to $G$-orbits are {\it vertical}.

 A $G$-invariant Riemannian structure on $M$ is {\it vertical} if its automorphism group leaves every $G$-orbit invariant. We denote by
 \begin{gather*}
 \mathcal{M}_G^v(M)\subset \mathcal{M}_G(M)
 \end{gather*}
 the space of $G$-invariant vertical Riemannian metrics.
\end{Definition}

Note that $\mathcal{M}(M)$ is open in the {\it Polish} (i.e.,~separable completely metrizable) space $\Gamma(T^{\otimes 2}M)$ of smooth sections of the tensor square bundle $T^{\otimes 2}M$, and hence is itself Polish~\cite[Appen\-dix, Proposition~A.1]{trev}. $\mathcal{M}_G(M)$ is also Polish, for instance because it is closed in the Polish space~$\mathcal{M}(M)$~\cite[Chapter~IX, Section~6.1, Proposition~1a)]{bour-top-2}.

Recall:

\begin{Definition}\label{def:cat1}
 A subset of a topological space is
 \begin{itemize}\itemsep=0pt
 \item {\it meager} or {\it of first category} if it is a countable union of nowhere dense sets,
 \item {\it non-meager} or {\it of second category} if it is not meager,
 \item {\it residual} if its complement is meager.
 \end{itemize}

 A topological space is {\it Baire} (or a {\it Baire space}) if meager sets have empty interior.
\end{Definition}
Cf.~\cite[Definitions 11.6.1 and 11.6.5]{nb}.

According to the Baire category theorem~\cite[Theorem~11.7.2]{nb} complete metric spaces are Baire. Since $\mathcal{M}_G(M)$ is Polish, that result allows us to regard residual subsets thereof as ``large'': they are certainly dense, but being residual says more than that.

As usual (e.g.,~\cite[Section~1.3]{eng-top}) {\it $F_{\sigma}$-subsets} of a topological space are countable union of~clo\-sed subsets, while {\it $G_{\delta}$-subsets} are countable intersections of open subsets. Their relevance here stems from the fact that in a Baire space a countable intersection of open dense subsets is residual and hence dense.

\section{Vertical metrics}\label{se:vert}

The reader might find~\cite{mb-gps,bred-gps,kob-gps} particularly useful in parsing the material below, on Lie-group actions on manifolds.

Theorem~\ref{th:equirig} below is an equivariant version of the result that ``generically'', Riemannian manifolds are rigid (i.e.,~have trivial isometry groups); this is~\cite[Proposition~8.3]{eb}, which can be recovered from Theorem~\ref{th:equirig} by setting $G=\{1\}$. Recall Definition~\ref{def:vert} above for notation. We also need the following notion (see, e.g.,~\cite[Proposition~I.2.5 and the discussion following it, and Remark I.2.7]{aud}).

\begin{Definition}\label{def:princ}
 Let $G$ be a compact Lie group acting smoothly on a smooth manifold $M$ so that $M/G$ is connected.
 \begin{enumerate}[($a$)]\itemsep=0pt
 \item An orbit $Gp$ is {\it principal} if either of the two following equivalent conditions holds:
 \begin{itemize}\itemsep=0pt
 \item the points $q\in M$ whose isotropy groups $G_q$ are in the same conjugacy class as $G_p$ form a dense open subset of $M$,
 \item the action of $G_p$ on the quotient $T_pM/T_p(Gp)$ of tangent spaces is trivial.
 \end{itemize}
 \item The action is {\it principal} if all of its orbits are.
 \item In general (i.e.,~for possibly-disconnected $M/G$), the {\it components} of the action are the actions of $G$ on the preimages of the connected components of $M/G$. Every orbit is an orbit of some component, and hence the notion of principality makes sense for orbits in full generality.
 \item Similarly, a general action is principal if its components are.
 \end{enumerate}
\end{Definition}

\begin{Theorem}\label{th:equirig}
 Let $M$ be a compact smooth manifold equipped with a principal $G$-action by a~compact Lie group $G$. If
 \begin{gather}\label{eq:dimbd}
 \dim M_i\ge 3+\dim G
 \end{gather}
 for every connected component $M_i\subset M$ then the space $\mathcal{M}^v_G(M)$ of $G$-invariant vertical Riemannian metrics on $M$ in the sense of Definition~$\ref{def:vert}$ is a dense $G_{\delta}$-subset of $\mathcal{M}_G(M)$.
\end{Theorem}

\begin{Remark}
 As observed in the statement of~\cite[Proposition~8.3]{eb}, some requirement (\ref{eq:dimbd}) on dimensions is necessary: when $G$ is trivial, the circle $M\cong \mathbb{S}^1$ has isometry group $O(2)$ for~any Riemannian metric. Since in this case `vertical' simply means `with trivial automorphism group', there are no vertical metrics at all.

 This example rules out a $1+$ correction term in (\ref{eq:dimbd}), but not a $2+$ correction term; indeed, I do not know whether the result is sharp in this sense.
\end{Remark}

\begin{proof}[Proof of Theorem~\ref{th:equirig}: the $\boldsymbol{G_{\delta}}$ claim]
 Let $\cat{orb}_i$, $i\in \mathbb{Z}_{\ge 0}$ be a countable set of orbits that is dense in $M/G$, and $U_n$, $n\in \mathbb{Z}_{\ge 0}$ a countable set of $G$-invariant open subsets of $M$ which constitute fundamental systems of neighborhoods of the $\cat{orb}_i$. Set
 \begin{gather*}
 \mathcal{F}_{i,n}:=\{g\in \mathcal{M}_G(M)\,|\, \text{the automorphism group }\cat{aut}(g)\text{ moves some point of }\cat{orb}_i
 \\ \hphantom{\mathcal{F}_{i,n}:=\{}
 \text{ out of }U_n\}.
 \end{gather*}
 We now have
 \begin{enumerate}[$(i)$]
 \item\label{item:23} {\it Each $\mathcal{F}_{i,n}$ is closed.} To sketch this briefly, let
 \begin{gather*}
 \mathcal{F}_{i,n}\ni g_{\alpha}\underset{\alpha}{\longrightarrow} g\in \mathcal{M}_G(M)
 \end{gather*}
 be a convergent net. By assumption, each $g_{\alpha}$ admits an automorphism $\gamma_{\alpha}$ that moves some point $p_{\alpha}\in \cat{orb}_i$ outside $U_n$. Because for large enough $\alpha$ the metrics $g_{\alpha}$ are uniformly close to $g$, so are the global geodesic metrics they induce on $M$, and hence the union of all isometry groups of $g_{\alpha}$ (again, assuming large $\alpha$) will be {\it equicontinuous}~\cite[Section~45, Definition]{mun}.

 It now follows from the Arzela--Ascoli theorem (e.g., in the variant appearing as~\cite[Theorem~47.1]{mun}) that $\{\gamma_{\alpha}\}$ is relatively compact, as is $\{\gamma_{\alpha}^{-1}\}$ in the uniform topology on self-maps of $M$. This implies that we can find a subnet $\gamma_{\beta}$ convergent to an isometry $\gamma$ of $g$. Further passing to a subnet thereof if necessary, we can furthermore assume that $p_{\beta}$ converges to some point $p$ which of course still belongs to the closed set $\cat{orb}_i$. Finally,
 \begin{gather*}
 M\setminus U_n\ni \gamma_{\beta} p_{\beta} \to \gamma p
 \end{gather*}
 and hence this latter point will again belong to the closed set $M\setminus U_n$.
 \item\label{item:24} {\it The complement $\mathcal{M}_G(M)\setminus \mathcal{M}^v_G(M)$ is the union of the $\mathcal{F}_{i,n}$.}
 \end{enumerate}
 Jointly, ({\it \ref{item:23}}) and ({\it \ref{item:24}}) imply the desired $G_{\delta}$-ness conclusion.
\end{proof}

The proof of (the rest of) Theorem~\ref{th:equirig} will require some preparation, in part to recall, somewhat informally, the proof strategy for~\cite[Proposition~8.3]{eb}. That proof proceeds as follows.

\begin{enumerate}[$(a)$]\itemsep=0pt
\item
An arbitrary Riemannian metric $g_{ij}$ on $M$ is first perturbed slightly so that the maximum over $p\in M$ of the largest eigenvalue
 \begin{gather*}
 \max~\cat{spec}(\cat{ric}(p))
 \end{gather*}
 of the symmetric operator
 \begin{gather*}
 \cat{ric}(p):=R^i_j(p)\colon T_pM\to T_pM
 \end{gather*}
 is achieved at a unique point $p$, and the perturbation is confined to an arbitrarily small neighborhood $U$ of $p$.
\item
With this in hand, every isometry of $M$ with respect to the new metric will fix that unique point $p$.
\item
The procedure is repeated on small spheres around $p$ avoiding $U$, ensuring that the maximal eigenvalue of $R^i_j$ on such a sphere is achieved at a unique point, which will then again be fixed by every Riemannian isometry.
\item
Repeating the procedure a large (but finite) number of times, one obtains a metric whose isometry group fixes at least $\dim M+1$ ``independent'' points of $M$. It then follows that the isometry group must be trivial (e.g.,~\cite[Theorem~3]{ms}).
\end{enumerate}

The proof of Theorem~\ref{th:equirig} appearing below follows essentially the same plan, with some modifications. For one thing, in place of the maximal eigenvalue we consider other numerical invariants of a self-adjoint operator on a (real) Hilbert space:

\begin{Notation}
 Let $T\colon \mathbb{R}^n\to \mathbb{R}^n$ be a symmetric operator. We write
 \begin{itemize}\itemsep=0pt
 \item $\|T\|$ for its {\it norm} with respect to any real Hilbert space structure on $\mathbb{R}^n$; it is the largest $|\lambda|$ for $\lambda$ ranging over the spectrum $\cat{spec}(T)$.
 \item $\cat{spr}(T)$ for the {\it spread} of $T$, i.e.,~the length of the smallest interval containing $\cat{spec}(T)$.
 \end{itemize}
\end{Notation}

We will be interested in maximizing the norm or spread of the operators $\cat{ric}(p)$ instead. Note that in general, for a Riemannian manifold $M$,
\begin{gather*}
\max_{p\in M}~\cat{spr}(\cat{ric}(p))=0
\end{gather*}
precisely when each operator $\cat{ric}(p)$ is a scalar multiple of the identity or, equivalently, the Ricci $(2,0)$-tensor $R_{ij}$ is a ``conformal multiple'' of the metric $g_{ij}$:
\begin{gather}\label{eq:confeins}
 \forall\, p\in M,\qquad R_{ij}(p) = f(p) g_{ij}(p)
\end{gather}
for some function $f\colon M\to \mathbb{R}$. Assuming $M$ is connected, this is
\begin{itemize}\itemsep=0pt
\item no restriction at all when $\dim M=2$ (i.e.,~it is automatic)~\cite[Remark 1.96(a)]{bes},
\item equivalent to $M$ being an {\it Einstein manifold} when $\dim M \ge 3$, i.e.,~the function $f$ in (\ref{eq:confeins}) is in fact constant~\cite[Theorem~1.97]{bes}.
\end{itemize}

\begin{Notation}\label{not:ncr}
 For a smooth manifold $M$ equipped with a smooth action by a Lie group $G$ we introduce the following notation.
 \begin{itemize}\itemsep=0pt
 \item $\mathcal{NSR}_G(M)\subset \mathcal{M}_G(M)$ is the set of $G$-invariant Riemannian structures satisfying
 \begin{gather*}
 \cat{spr}(\cat{ric}(p))>0 \qquad\text{over a dense set of}\quad p\in M.
 \end{gather*}
 The symbol stands for ``non-scalar Ricci'', based on the fact that $\cat{spr}(\cat{ric}(p))$ vanishes precisely when the operator $\cat{ric}(p)\colon T_pM\to T_pM$ is a scalar multiple of the identity.
 \item Similarly, $\mathcal{NZR}_G(M)$ (for ``non-zero'') is the set of $G$-invariant Riemannian structures such that
 \begin{gather*}
 \cat{ric}(p)\ne 0 \qquad\text{over a dense set of}\quad p\in M.
 \end{gather*}
 \item $\mathcal{NZS}_G(M)$ (for ``non-zero scalar'') is the set of $G$-invariant Riemannian structures such that
 \begin{gather*}
 R(p)=\mathrm{tr}(\cat{ric}(p))\ne 0 \qquad\text{over a dense set of}\quad p\in M.
 \end{gather*}
 \item For a subset $U\subseteq M$, we write $\mathcal{SR}^U_G(M)$ for the set of $G$-invariant structures for which
 \begin{gather*}
 \cat{spr}(\cat{ric}(p))=0,\qquad \forall\, p\in U
 \end{gather*}
 (i.e.,~the Ricci tensor is scalar along $U$).
 \item Finally, we set
 \begin{gather*}
 \mathcal{SR}_G(M):=\mathcal{SR}_G^M(M).
 \end{gather*}
 \end{itemize}
\end{Notation}

\begin{Proposition}\label{pr:ncrgen}
 Let $M$ be a smooth compact manifold with connected components $M_i$, and equipped with a smooth action by a compact Lie group $G$.
 \begin{enumerate}[$(a)$]\itemsep=0pt
 \item\label{item:5} If
 \begin{gather*}
 \dim M_i\ge 2+\dim G,\qquad \forall\, i
 \end{gather*}
 the set $\mathcal{NZR}_G(M)$ is residual in $\mathcal{M}_G(M)$ in the sense of Definition~$\ref{def:cat1}$.
 \item\label{item:7} The same goes for $\mathcal{NZS}_G(M)$.
 \item\label{item:6}
 If furthermore we have $\dim M_i\ge 3$ for all $i$ then the space $\mathcal{NSR}_G(M)$ is residual in $\mathcal{M}_G(M)$ in the sense of Definition~$\ref{def:cat1}$.
 \end{enumerate}
\end{Proposition}
\begin{proof}
 We prove ({\it \ref{item:6}}), while only very briefly sketching how the (simpler) proofs for parts ({\it \ref{item:5}}) and ({\it \ref{item:7}}) can be adapted from this.

 ({\it \ref{item:6}}) The complement
 \begin{gather*}
 \mathcal{M}_G(M)\setminus \mathcal{NSR}_G(M)
 \end{gather*}
 is the union, over all open $U\subseteq M$, of the sets $\mathcal{SR}^U_G(M)$ introduced in Notation~\ref{not:ncr}. Since we can furthermore range $U$ over some countable base for the topology of $M$, it will be enough to prove that for every non-empty open $U$ the set $\mathcal{SR}^U_G(M)$ is nowhere dense in $\mathcal{M}_G(M)$. Since that set is closed, what we want to argue is that it has empty interior. In other words:

 \begin{Claim}
A metric $g\in \mathcal{SR}^U_G(M)$ has arbitrarily small deformations outside that set.
 \end{Claim}

 We can see this by effecting a conformal deformation
 \begin{gather*}
 g\mapsto g':=\varphi^{-2} g,
 \end{gather*}
 where $\varphi$ is a strictly positive, $G$-invariant function on $M$ that is $C^{\infty}$-close to the constant function $1$.

 We can assume that $U$ is $G$-invariant. According to the slice theorem for $G$-actions (e.g.,~\cite[Th\'eor\`eme, p.~139]{ksz},~\cite[Theorem~1]{my} or~\cite[Theorem~I.2.1]{aud} among others) every point $p\in U$ has a~$G$-invariant ``tubular'' neighborhood contained in $U$, $G$-equivariantly diffeomorphic to $G\times_{G_p}V$, where
 \begin{itemize}\itemsep=0pt
 \item $G_p\subseteq G$ is the isotropy group at $p$,
 \item $V$ is the quotient space $T_pM/T_p (Gp)$ ($Gp$ being the orbit through $p$),
 \item the $G_p$-action on $V$ is the differential of the $G_p$-action on $M$ obtained by restricting that of $G$.
 \end{itemize}
 Furthermore, it follows from~\cite[Proposition~I.2.5]{aud} that there is a dense set of points $p$ for which the linear action of $G_p$ on $V$ is trivial (i.e.,~those lying on principal orbits in the sense of~Definition~\ref{def:princ}). For such a $p\in U$ (which we henceforth fix), the tubular neighborhood $G\times_{G_p}V$ is in fact diffeomorphic to the product manifold $Gp\times V$; we frequently identify the two in the discussion below. We can then select our scaling function $\varphi$ so that
 \begin{itemize}\itemsep=0pt
 \item it is identically $1$ outside some $G$-invariant neighborhood of $Gp$ whose closure is contained in $Gp\times V$,
 \item on $Gp\times V$ it depends only on local coordinates on $V$, and is thus $G$-invariant.
 \end{itemize}

 Additionally, we have to choose $\varphi$ so as to achieve the desired outcome that $g'$ have non-scalar Ricci $(1,1)$-tensor in $U$. By~\cite[equation (1.161b)]{bes} the conformal transformation rules for the traceless Ricci tensor (\ref{eq:trric}) are of the form
 \begin{gather}\label{eq:cnfrmz}
 Z' = Z + (\text{some multiple of }g) + \frac{\dim M-2}{\varphi}\operatorname{Hess}(\varphi),
 \end{gather}
 where $\operatorname{Hess}$ denotes the {\it Hessian} defined \cite[Section~1.54]{bes} as a $(2,0)$-tensor by
 \begin{gather*}
 \operatorname{Hess}(\varphi)(X,Y) = X(Y\varphi) - (\nabla_XY)(\varphi),
 \end{gather*}
 where $\nabla$ denotes the Levi-Civita connection (denoted by the same symbol in~\cite[Section~IV.2]{kn1} and by $D_{X}Y$ in~\cite[Section~1.41]{bes}).

 Since $\dim M\ge 3$, it will be enough to choose $\varphi$ so that Hessian fails to be a scalar multiple of the metric $g$ at some point in $U$. In {\it normal} \cite[Section~1.44]{bes} local coordinates $\operatorname{Hess}(\varphi)$ is expressible as the familiar Hessian matrix with entries
 \begin{gather}\label{eq:2ndders}
 \operatorname{Hess}(\varphi)_{i,j} = \frac{\partial^2 \varphi}{\partial x^i \partial x^j}.
 \end{gather}
 Since (with $M_i\subset M$ being the component that contains $p$) we have
 \begin{gather*}
 \dim V = \dim M_i - \dim Gp\ge \dim M_i-\dim G\ge 2
 \end{gather*}
 by assumption, we can certainly arrange for second partial derivatives with respect to the coordinates $x^i$ on $V$ so that the bilinear form with matrix (\ref{eq:2ndders}) is not a scalar multiple of $(g_{ij})_{i,j}$. This proves the claim and hence the result.

 ({\it \ref{item:7}}) We can follow the same strategy as above, this time replacing (\ref{eq:cnfrmz}) with its scalar-curvature version \cite[Theorem~1.159(f)]{bes}: if we conformally scale the metric $g$ to $g'=e^{2f}g$ then the relation between the two scalar curvatures $R'$ (new) and $R$ (old) is
 \begin{gather*}
 R' = e^{-2f}\left(R + 2(n-1)\Delta f - (n-2)(n-1)|{\rm d}f|^2\right),
 \end{gather*}
 where
 \begin{itemize}\itemsep=0pt
 \item $n$ is the dimension of the underlying manifold,
 \item $\Delta f$ is the {\it Laplacian} of $f$~\cite[Section~1.54c]{bes},
 \item $|{\rm d}f|$ denotes the length of the {\it gradient} of $f$~\cite[Section~1.54a]{bes} in the metric $g$.
 \end{itemize}

 ({\it \ref{item:5}}) This follows from parts ({\it \ref{item:7}}) and ({\it \ref{item:6}}): the former trivially covers components of dimen\-sion~\mbox{$\ge 3$}, whereas the latter ensures non-vanishing on dimension-2 components, where $R_{ij} = \frac 12 Rg_{ij}$.
\end{proof}

Lemma~\ref{le:1invorb} implements ({\it \ref{item:5}}) (and ({\it \ref{item:7}})) in the above discussion, following the statement of Theo\-rem~\ref{th:equirig}; its proof is very much in the spirit of that of~\cite[Proposition~8.3]{eb}.

\begin{Lemma}\label{le:1invorb}
 Let $G$ be a compact Lie group acting smoothly and isometrically on a Riemannian manifold $(M,g)$ with components of dimension $\ge 2+\dim G$. Then, there is a point $p\in M$ such that
 \begin{itemize}\itemsep=0pt
 \item one can find $G$-invariant metrics $g'$ on $M$ arbitrarily close to $g$,
 \item achieving the maximal absolute value of its scalar curvature on a unique $G$-orbit in an~arbi\-trarily small $G$-invariant neighborhood $U$ of $Gp$, and hence,
 \item so that the isometry group $\cat{aut}(g')$ leaves that orbit invariant.
 \end{itemize}
 Moreover, if $g\not\in \mathcal{SR}_G(M)$ then we can ensure $g'= g$ outside the arbitrarily-small neighborhood~$U$ of $Gp$.
\end{Lemma}
\begin{proof}
 By part ({\it \ref{item:7}}) of Proposition~\ref{pr:ncrgen} we can perturb $g$ (arbitrarily) slightly so as to ensure the scalar curvature
 \begin{gather*}
 R(p) = \operatorname{tr} \cat{ric}(p)
 \end{gather*}
 is non-zero for most $p$. We retain this assumption on $g$ throughout the rest of the proof.

 Now let $p\in M$ be a point where the maximal absolute value $|R(q)|$, $q\in M$ is achieved (it~will be the point $p$ in the statement), and fix a $G$-invariant neighborhood $U$ of $Gp$. 
 Consider a smooth function
 \begin{gather*}
 \psi\colon\ \mathbb{R}_{\ge 0}\to \mathbb{R}_{\ge 1}
 \end{gather*}
 that is
 \begin{itemize}\itemsep=0pt
 \item $C^{\infty}$-close to the constant function $1$,
 \item equal to some constant slightly larger than $1$ on a small interval $[0,r]$,
 \item equal to $1$ on $[r+\varepsilon,\infty)$.
 \end{itemize}
 One then obtains a smooth $G$-invariant function $\varphi$ on $M$, $C^{\infty}$-close to $1$, by
 \begin{gather*}
 \varphi(x):=\psi(\text{distance from $x$ to the orbit $Gp$}),\qquad \forall\, x\in M.
 \end{gather*}
 We assume $r$ in the above discussion is small enough that $\varphi$ is identically $1$ off $U$.

 Finally, consider the $G$-invariant conformal rescaling $g_1:=\varphi^{-2} g$. Because it scales $g$ by the constant $\psi(0)^{-2}<1$ in a neighborhood of $Gp$, it scales the operator $\cat{ric}(p)$ (and hence its trace) by the inverse scalar $\psi(0)^2>1$. Since $g_1\cong g$ off $U$, the new metric achieves its maximal
 \begin{gather}\label{eq:sprric}
 |R(q)|,\qquad q\in M
 \end{gather}
 somewhere in $U$.

 Now repeat the procedure, as in the proof of~\cite[Proposition~8.3]{eb}: pick $q\in U$ maximizing~(\ref{eq:sprric}) for $g_1$, choose a neighborhood $U_1$ of $q$ less than half the size of $U$ with respect to some fixed metric inducing the topology of $M$, and perturb $g_1$ to $g_2$ so that
 \begin{itemize}\itemsep=0pt
 \item the perturbation $g_2-g_1$ is less than half the size of $g_1-g$ in some metric inducing the $C^{\infty}$ topology on the space of Riemannian structures,
 \item $g_2=g_1$ off $U_1$, and
 \item for $g_2$ the maximal value of (\ref{eq:sprric}) is achieved in $U_2$.
 \end{itemize}
 Continuing in this fashion, the limit
 \begin{gather*}
 g':=\lim_{n\to\infty}g_n
 \end{gather*}
 will be a $G$-invariant metric close to $g$ whose maximal (\ref{eq:sprric}) is achieved on a {\it unique} orbit contained in the original (arbitrarily small) neighborhood $U$ of $p$. It follows that orbit must be preserved by the isometry group of $g'$, as desired.

As for the last statement (on $g\not\in \mathcal{SR}_G(M)$), it is clear from the proof: the argument produces metrics identical to $g$ off $U$ after the initial step of perturbing $g$ away from $\mathcal{SR}_G(M)$.
\end{proof}

\begin{proof}[Proof of Theorem~\ref{th:equirig}]
 By passing to the components of the action in the sense of Definition~\ref{def:princ}, we may as well assume that the orbit space $M/G$ is connected. Furthermore, by~Lemma~\ref{le:1invorb} we can assume that our metric $g$ achieves its maximal scalar curvature along a single orbit $Gp$ (for some $p\in M$).

 Now consider the geodesics emanating from $p$, orthogonal to $Gp$ (we refer to such geodesics as {\it horizontal}, in keeping with the spirit of Definition~\ref{def:vert}). Denoting by $d_g$ the distance induced by the metric $g$, for sufficiently small $r>0$ the tubular neighborhood
 \begin{gather*}
 Gp_{\le r}:=\{q\in M\,|\, d_g(q,Gp)\le r\}
 \end{gather*}
 is (by the principality of the action) diffeomorphic to $Gp\times H_{\le r}$, where
 \begin{itemize}\itemsep=0pt
 \item $H$ is the union of the horizontal geodesics emanating from $p$, and hence a manifold close enough to $Gp$,
 \item $H_{\le r}$ is, as the notation suggests, the subset of $H$ at distance $d_g\le r$ from the orbit $Gp$ (or~equivalently, from $p$).
 \end{itemize}
 Horizontal geodesics are orthogonal to all $G$-orbits they encounter (e.g.,~\cite[Lemma 9.44]{bes} or~\cite[Section~1.1]{haef-riem}), and we can obtain $G$-invariant Riemannian structures by deforming the metric~$g$ along the manifold~$H$ comprising the horizontal geodesics (sufficiently close to $G_p$ so as not to run into injectivity-radius issues) and keeping it invariant along the $G$-orbits. Explicitly, at a~point $q\in H_r$ we can split the tangent space $T_qM$ as
 \begin{gather*}
 T_qM = T_q(Gq)\oplus T_qH,
 \end{gather*}
 decompose the matrix of the Riemannian metric $g$ correspondingly as a block matrix
 \begin{gather*}
 \begin{pmatrix}
 A_v&B\\
 B^t&A_h
 \end{pmatrix}
 \end{gather*}
 (with the top left and bottom right corners representing, respectively, the restrictions of $g$ to $Gq$ and $H$), and deforming only the lower right-hand corner $A_h$ sufficiently slightly so as to ensure the resulting matrix still represents a {\it positive} symmetric bilinear form.

 The isometry group $\cat{aut}(g)$ leaves $Gp$ invariant, and hence the isotropy subgroup $\cat{aut}(g)_p$ preserves every $p$-centered ball $H_{\le r}$ in $H$. Now choose small $r,\varepsilon>0$ and deform the metric slightly in $H_{\le 2r}$ so that
 \begin{itemize}\itemsep=0pt
 \item the perturbed metric coincides with the old metric $g$ outside $H_{\le r+\varepsilon}$ and inside $H_{\le r-\varepsilon}$,
 \item inside the annulus $H_{[r-\varepsilon,r+\varepsilon]}$ the perturbation is {\it spherical}, in the sense that we choose {\it geodesic spherical coordinates} \cite[Section~III.1]{chav} in $H_{\le r}$ centered at $p$, with a radial coor\-dinate and $(\dim H-1)$ ``angular'' coordinates, and deform the metric only along the latter,
 \item the perturbed metric on the sphere $H_r$ has trivial isometry group (this is possible because that sphere is at least $2$-dimensional by (\ref{eq:dimbd}), and hence~\cite[Proposition~8.3]{eb} applies).
 \end{itemize}
 For the resulting $G$-invariant metric $g'$ the manifold $H$ consisting of horizontal geodesics emanating from $p$ still bears that description because of the spherical character of the deformation. By construction, the isotropy group $\cat{aut}(g')_p$ will then fix $H_{r}$ identically (i.e.,~pointwise). But~in that case
\begin{itemize}\itemsep=0pt
\item $\cat{aut}(g')_p$ leaves invariant the $G$-orbit of every point in the tubular neighborhood $GH_{r}$ of~$Gp$,
\item and hence so does
 \begin{gather*}
 \cat{aut}(g') = G \cdot \cat{aut}(g')_p.
 \end{gather*}
\end{itemize}
Note that the latter product is not direct, and `$G$' is a stand-in for its image in the automorphism group of $g'$ (the action of $G$ is not assumed faithful here).

Since we are assuming the orbit space $M/G$ is connected, all orbits are reachable from $Gp$ by~horizontal geodesics emanating from it. Since $\cat{aut}(g')_p$ acts trivially on the initial segments of~those geodesics it acts trivially on horizontal geodesics period, meaning that {\it all} orbits are left invariant by $\cat{aut}(g')$.
\end{proof}

\subsection{Some remarks on the literature}

The discussion above gives a brief review of the proof of~\cite[Proposition~8.3]{eb}. For this reader, at least, that proof presented a difficulty that appeared not to be immediately addressed by the text in loc.cit. Specifically, the proof proceeds, as indicated above, by
\begin{enumerate}[(1)]\itemsep=0pt
\item\label{item:18} first deforming a metric $g$ so as to produce a {\it globally-invariant} point $p$ (i.e.,~one fixed by all isometries), and then
\item\label{item:19} deforming the metric again around a radius-$r$ sphere $S_{p,r}$ centered at $p$ so as to produce a~point $q$ where the Ricci $(1,1)$-tensor $\cat{ric}$ achieves its unique maximal spectral value along~$S_r(p)$.
\end{enumerate}

The isometry group of the metric obtained after step (\ref{item:18}) will leave $p$ invariant, and hence also $S:=S_r(p)$ (which in~\cite{eb} would be denoted by $A_p^r$). If $\cat{ric}$ were to achieve its maximal spectral value at a unique point $q\in S$ at this stage, then $q$ would be invariant under the isometry group. The problem, though, is that $q$ is produced after further deformation, whereupon $S$ need not remain a $p$-centered sphere.

In other words, I see no reason (without further elaboration) why the metric produced after~(\ref{item:19}) should leave $S$ invariant (and hence $q$ on it). There are ways to handle this:
\medskip

{\bf Deforming outside a ball.} The alteration of the metric ``around $S$'' (as it is phrased on~\cite[p.~36]{eb}, with $A_q^{\rho}$ in place of $S$) might be interpreted as an alteration only outside the ball $B_r(p)$ bounded by $S$. This is possible, since the alteration in question consists of adding to the $(2,0)$-tensor $g$ another tensor whose $2^{nd}$ derivatives with respect to a system of normal coordinates satisfy certain inequalities (see~\cite[equation (8.4)]{eb}).

This would ensure that after the deformation in (\ref{item:19}) the radius $r$-sphere centered at $p$ retains its identity.

\medskip

{\bf An inductive approach.} Alternatively, one could proceed inductively on dimension, by
\begin{itemize}\itemsep=0pt
\item first proving the claim separately for surfaces, and then
\item finding $p$ as above, and then modifying the metric only on geodesic spheres around $p$ as in the proof of Theorem~\ref{th:equirig}, making use of spherical coordinates.
\end{itemize}

\section{Maximal rigidity}\label{se:maxrig}

As indicated in the Introduction, the initial motivation for the results above was to produce $G$-invariant metrics whose isometry group is {\it precisely} $G$; they should, in other words, be maximally rigid subject to the requirement that they be $G$-invariant (hence the title of the present section). This also justifies

\begin{Notation}
 Given a faithful isometric action of a Lie group $G$ on a Riemannian manifold~$M$, the space $\mathcal{M}_G^{\max}(M)$ of {\it maximally rigid} $G$-invariant metrics consists of those $g\in \mathcal{M}_G(M)$ whose isometry group is precisely $G$.

 The same notation (and terminology) applies to arbitrary (non-faithful) actions: if $H\trianglelefteq G$ is the kernel of the action, then by definition
 \begin{gather*}
 \mathcal{M}_G^{\max}(M) = \mathcal{M}_{G/H}^{\max}(M).
 \end{gather*}
\end{Notation}
Since we can harmlessly pass to faithful actions by passing to the quotient by the kernel of the action, we typically assume faithfulness throughout.

One cannot hope for metrics produced as in Theorem~\ref{th:equirig} to be maximally rigid in full generality, for arbitrary compact Lie groups. Indeed, most {\it finite} groups $G$ will fail in that respect:

\begin{Example}\label{ex:disc}
 Let $G$ be a compact Lie group with $\ge 3$ connected components $G_i$, acting in the obvious fashion on $M:=G\times N$ for some manifold $N$. Then, for any $G$-invariant Riemannian structure $g$ on $M$, the automorphism group $\cat{aut}(g)$ can permute the manifolds $G_i\times N$ for $\gamma\in N$ arbitrarily.

Now, if $G_0\subset G$ is the identity component, then the action of $G$ on the set of manifolds $G_i\times N$ is isomorphic (as a permutation action) to the regular action of $G/G_0$. Since the latter is strictly smaller than the symmetric group $S(G/G_0)$ of the {\it set} $G/G_0$, we have
 \begin{gather*}
 S(G/G_0)\subset \cat{aut}(g)\quad \text{but}\quad S(G/G_0)\not\subseteq G\subset \cat{aut}(g).
 \end{gather*}
 In particular, for such $G$ (and actions) we can never obtain $G=\cat{aut}(g)$ for a suitable Riemannian metric $g$.
\end{Example}

It turns out, though, that the disconnectedness of $G$ in Example~\ref{ex:disc} is the {\it only} issue:

\begin{Theorem}\label{th:precaut}
 Let $G$ be a compact connected Lie group acting freely and smoothly on a compact smooth manifold $M$. Then, the following statements hold.
 \begin{itemize}\itemsep=0pt
 \item The subset
 \begin{gather}\label{eq:maxrigmetrics}
 \mathcal{M}^{\max}_G(M) \subseteq \mathcal{M}_G(M)
 \end{gather}
 is open.
 \item If furthermore the components $M_i$ of $M$ satisfy the dimension inequality
 \begin{gather}\label{eq:2g}
 \dim M_i\ge \max(3+\dim G,2\dim G+1)
 \end{gather}
 then $(\ref{eq:maxrigmetrics})$ is dense.
 \end{itemize}
\end{Theorem}

As an immediate consequence we have

\begin{Corollary}\label{cor:isiso}
 Every compact connected Lie group arises as the isometry group of some compact Riemannian manifold.
\end{Corollary}
\begin{proof}
 In Theorem~\ref{th:precaut}, simply take $M=G\times N$ equipped with the obvious action on the left-hand factor for some connected manifold $N$ of sufficiently large dimension.
\end{proof}

This answers the question in~\cite[Section~4]{mell} (and~\cite[Question~Q4]{niem}) affirmatively.

\begin{Remark}\label{re:conn}
Since $G$ is connected, it operates on each connected component of $M$. Restricting our attention to an individual component, we can assume that $M$ is connected; we do this throughout the present section. With this connectedness assumption in place, an isometry of~$M$ is trivial if and only if
\begin{itemize}\itemsep=0pt
\item it fixes some point $p$ (arbitrary, chosen beforehand), and
\item it induces the trivial linear action on $T_pM$.
\end{itemize}
\end{Remark}

\begin{proof}[Proof of Theorem~\ref{th:precaut}: openness]
 This follows from the upper semicontinuity of the automorphism group of Riemannian structures. Let $g\in \mathcal{M}_G^{\max}(M)$. According to the aforementioned semicontinuity result \cite[Theorem~8.1]{eb}, for $g'\in \mathcal{M}_G(M)$ sufficiently close to $g$ we have
 \begin{gather}\label{eq:conjin}
 \sigma \cat{aut}(g') \sigma^{-1}\subseteq \cat{aut}(g) = G
 \end{gather}
 for some diffeomorphism $\sigma$ of $M$. The left hand side is a subgroup of $\cat{diff}(M)$ (group of~dif\-feomorphisms) containing the Lie group
 \begin{gather*}
 \sigma G\sigma^{-1} \subset \cat{diff}(M)
 \end{gather*}
 because, $g'$ being $G$-invariant, $\cat{aut}(g')$ contains $G$. Since Lie groups cannot contain proper isomorphic copies of themselves (\ref{eq:conjin}) must be an equality. It follows that so too is
 \begin{gather*}
 G\subseteq \cat{aut}(g'),
 \end{gather*}
 again for reasons of size: $G$ and $\cat{aut}(g')$ are Lie groups with the same dimension and the same number of components, one containing the other.
\end{proof}

We have the following characterization of maximally rigid actions.

\begin{Lemma}\label{le:charmax}
 A vertical free action of a compact Lie group $G$ on a connected manifold $M$ is maximally rigid if and only if either of the following equivalent statements holds:
 \begin{enumerate}[$(a)$]\itemsep=0pt
 \item
 the action of the isometry group is free, i.e.,~the isotropy group of every point is trivial,
 \item\label{item:9}
 the isotropy group of a {\it single} arbitrary point $p\in M$ is trivial.
 \end{enumerate}
\end{Lemma}
\begin{proof}
 We only prove equivalence to ({\it \ref{item:9}}), leaving the other point to the reader.

 For a vertical metric $g\in \mathcal{M}_G^v(M)$ an arbitrary point $p\in M$ will be moved by every isometry~$\sigma$ to a point $q$ on the same orbit $Gp$. We can then translate $q$ back to $p$ via the $G$-action, i.e.,~by some element $\gamma\in G$. Then, $\sigma$ belongs to $G$ if and only if $\gamma\sigma$ does. Since the action is free, the isotropy group $G_p$ is trivial. We already know that $\gamma\sigma$ is in the isotropy group $\cat{aut}(g)_p$, so
 \begin{gather*}
 \sigma\in G\iff \sigma\gamma\in G\iff \sigma\gamma\in G_p\iff \sigma\gamma=1.
 \end{gather*}
 Since, as $\sigma$ ranges over $\cat{aut}(g)$, elements of the form $\sigma\gamma$ range over $\cat{aut}(g)_p$, this proves the equivalence between maximal rigidity and ({\it \ref{item:9}}).
\end{proof}

We will often keep this characterization in mind in the arguments below, sometimes implicitly.

Note that even though Theorem~\ref{th:equirig} only says that the vertical metrics form a $G_{\delta}$ (rather than open) set, in the context of that proof we have quite a bit of freedom in varying $g$ so as to keep it vertical. Specifically, if, as in that proof, we assume the maximal scalar curvature is achieved along a unique orbit $Gp$ (as we will), then all metrics $g'$
\begin{itemize}\itemsep=0pt
\item sufficiently $C^{\infty}$-close to $g$,
\item coinciding with $g$ close to $Gp$
\end{itemize}
will be vertical. This is because, again as in the aforementioned proof, the corresponding ``horizontal'' manifold $H$ through $p$ (i.e.,~the union of the geodesics emanating from $p$ and orthogonal to $Gp$) will have trivial isometry group by~\cite[Proposition~8.3]{eb}. For these reasons, we need not worry below, in the proof of Theorem~\ref{th:precaut}, about breaking the verticality of our slightly-deformed Riemannian structures.

\begin{proof}[Proof of Theorem~\ref{th:precaut}: density]
 According to Theorem~\ref{th:equirig} we can deform an arbitrary metric arbitrarily slightly so as to render it vertical, so we work with vertical metrics $g$ to begin with. In fact, we will assume (via Lemma~\ref{le:1invorb}) that the maximal scalar curvature of $g$ is achieved along a unique orbit~$Gp$, and hence that orbit is left invariant.

 We also reprise some of the notation (and setup) from the proof of Theorem~\ref{th:equirig}: $H$ will be a~manifold consisting of sufficiently short geodesic arcs based at $p$ and orthogonal to the orbit~$Gp$, we work inside small balls $H_{\le r}$ therein, etc. When we want to indicate the dependence of~$H$ on~$g$ and/or $p$ we decorate $H$ with those subscripts, as in $H_g$, $H_p$ or, maximally, $H_{g,p}$.

 From Remark~\ref{re:conn} and Lemma~\ref{le:charmax} we know that it suffices to find metrics $g'$, close to $g$, for~which the isotropy group of some (or any) $q\in M$ acts trivially on the tangent space $T_qM$. The~isotropy group $\cat{aut}(g)_p$ of $p$ acts trivially on
 \begin{itemize}\itemsep=0pt
 \item the horizontal manifold $H_{g,p}$ at $p$ and hence on every $T_q(H_{g,p})$ for $q$ thereon,
 \item on the horizontal manifold $H_{g,q}$ at $q\in H_{g,p}$, if $q$ is sufficiently close to $p$, because in that case the restricted metric on $H_{g,q}$ will be close to that on its diffeomorphic counterpart~$H_{g,p}$, and hence will be rigid by~\cite[Corollary~8.2 and Proposition~8.3]{eb}.\footnote{The proof of~\cite[Proposition~8.3]{eb}, asserting density, does not require that the manifold be boundary-less. On the other hand, while the openness result~\cite[Corollary~8.2]{eb} is nominally proved for boundary-less manifolds (though see~\cite[p.~11, footnote~3]{eb}), one can simply regard $H_{g,p}$ as a subset of such a manifold: Riemannian structures always extend from compact manifolds with boundary to compact manifolds without boundary, e.g., by~\cite[Theorem~A]{pv}.}
 \end{itemize}

 \begin{Claim}
 $g$ can always be deformed slightly so as to ensure that for $q\in H_{g,p}$ close to $p$ the subspaces
 \begin{gather}\label{eq:2sp}
 T_q(Gq)^{\perp}\qquad\text{and}\qquad T_q(H_{g,p})\subset T_qM
 \end{gather}
 are in general position, i.e., intersect minimally.
 \end{Claim}

 {\bf Wrapping up assuming the claim.} Since
 \begin{itemize}\itemsep=0pt
 \item they always intersect at least along the line in $T_q(H_{g,p})$ tangent to the geodesic connecting~$p$ and $q$, and
 \item we have
 \begin{gather*}
 \dim T_q(Gq)^{\perp} + \dim T_q(H_{g,p}) -1 = 2(\dim M-\dim G)-1\ge \dim M
 \end{gather*}
 \end{itemize}
 (by (\ref{eq:2g})), general position means that
 \begin{gather*}
 T_q(Gq)^{\perp} + T_q(H_{g,p}) = T_qM.
 \end{gather*}
 In conclusion, upon performing a small deformation of $g$ the group $\cat{aut}(g)_p$ fixes $q$ and acts trivially on $T_qM$, and is thus trivial. The conclusion follows, finishing the proof of the theorem.

 {\bf Proof of the claim.} This asserted our ability to deform $g$ so as to have (\ref{eq:2sp}) placed in general position. To see this, note first that for any $g'\in \mathcal{M}_G(M)$ the map $\pi\colon M\to M/G$ is a Riemannian submersion in the sense of~\cite[Definition~9.8]{bes} and conversely (e.g., by~\cite[Section~9.15]{bes}), in order to specify a $G$-invariant metric on $M$ we need to fix
 \begin{itemize}\itemsep=0pt
 \item a Riemannian structure on $M/G$,
 \item smoothly-varying $G$-invariant Riemannian structures on the fibers (isomorphic to $G$) of $M\to M/G$,
 \item a $G$-invariant {\it distribution} $\mathcal{H}\subset TM$ (i.e.,~a smoothly-varying choice of subspaces $\mathcal{H}_x\subset T_xM$ for $x\in M$) complementary to the {\it vertical} distribution $\mathcal{V}$ consisting of vectors tangent to fibers.
 \end{itemize}
 ($\mathcal{H}_g$ will then consist of the tangent vectors orthogonal to the fibers.) Correspondingly, our desired modification of $g$ will
 \begin{itemize}\itemsep=0pt
 \item leave the already-existing Riemannian structure on $M/G$ unaffected,
 \item leave the already-existing metrics on the fibers unaffected,
 \item alter only the horizontal distribution $\mathcal{H}_g$ attached to $g$ slightly, to $\mathcal{H}_{g'}$.
 \end{itemize}

 Recall that $H$ consists of geodesics emitted from $p$ and orthogonal to $Gp$, and we chose $q\in H$ some small distance $r$ away from $p$. The tangent space $T_q(H_{g,p})$ is spanned by the line tangent to the geodesic $pq$ and the tangent space $T_q(H_{g,p,r})$ where, consistently with the notation $H_{\le r}$ above,
 \begin{gather*}
 H_{g,p,r}:=\{x\in H_{g,p}\,|\, d_g(p,x)=r\}
 \end{gather*}
 is the radius-$r$ sphere centered at $p$ along $H$. The line tangent to the geodesic $pq$ will always be orthogonal to $T_q(Gq)$ (a geodesic horizontal at one point is horizontal everywhere:~\cite[Lem\-ma~9.44]{bes}), but the crucial observation is that by deforming $g$ slightly, we can
 \begin{enumerate}[$(a)$]\itemsep=0pt
 \item\label{item:20} keep $T_q(Gq)^{\perp}$ invariant,
 \item\label{item:17} make $T_q(H_{g,p,r})$ sweep out an open subset of the relevant Grassmannian, hence the desired generic-position conclusion.
 \end{enumerate}

 To achieve these last two goals (({\it \ref{item:20}}) and ({\it \ref{item:17}})) note first that denoting as above by
 \begin{gather*}
 \pi\colon\ M\to M/G
 \end{gather*}
 the canonical projection, the geodesics $p\to x$ for $p$ to points $x\in H_{g,p,r}$ are the horizontal lifts of~the geodesics in $M/G$ connecting $\pi(p)$ to the points $\pi(x)$ on the radius-$r$ sphere $S_{\pi(p),r}$ around~it. Now choose any submanifold $S$ of $M$ that
 \begin{itemize}\itemsep=0pt
 \item is $C^{\infty}$-close to $H_{g,p,\le 2r}$ (in particular, it is transverse to the $G$-orbits $\le 2r$ away from $Gp$ and has the same dimension as $H_{g,p,\le 2r}$),
 \item is horizontal (i.e.,~orthogonal to the $G$-orbits) along the geodesic line connecting $p$ and $q$, and
 \item coincides with $H_{g,p,\le 2r}$ off $H_{g,p,\le r}$.
 \end{itemize}
 We can now declare the tangent spaces to $S$ to be horizontal (for a new metric $g'$ on $M$), obtaining a $G$-invariant distribution on the tubular neighborhood
 \begin{gather*}
 \{x\in M\,|\, d_g(x,Gp)\le 2r\}
 \end{gather*}
 by operating with $G$. Because we imposed the condition that $S=H_{g,p,\le 2r}$ off $H_{g,p,\le r}$, this glues with $g$ to obtain a globally-defined $G$-invariant metric $g'$ on $M$ that perturbs $g$ slightly.

 With respect to $g'$ the new horizontal lifts of the geodesics
 \begin{gather*}
 \pi(p)\to \pi(x)\in S_{\pi(x),r}
 \end{gather*}
 are their lifts to $S=H_{g',p,\le r}$ (rather than the old $H_{g,p,\le r}$). Clearly, this gives us sufficient freedom to move the tangent space $T_q(H_{g',p,r})$ within a small neighborhood of the old $T_q(H_{g,p,r})$, as desired.
\end{proof}

\subsection*{Acknowledgements}

This work is partially supported by NSF grants DMS-1801011 and DMS-2001128.
I am indebted to the anonymous referees for numerous suggestions contributing to the improved quality of the initial draft. In particular, I would have remained unacquainted with~\cite{bd,sz} were it not for one of the referee reports.

\pdfbookmark[1]{References}{ref}
\LastPageEnding

\end{document}